\documentclass{article}
\usepackage{amsmath,amssymb}
\usepackage[colorlinks=true]{hyperref}
\hypersetup{urlcolor=blue, citecolor=blue}

\def\div{{\rm div}}
\def\O{\Omega}
\def\OT{\Omega\times(0,T)}
\def\be{\begin{equation}}
\def\ee{\end{equation}}
\def\R{\mathbb{R}}
\def\sign{{\rm sign}}
\newtheorem{thm}{\bf Theorem}[section]

\newtheorem{prop}[thm]{\bf Proposition}

\newtheorem{Def}{\bf Definition}

\newtheorem{rmk}{\bf Remark}[section]

\newcommand{\CQFD}{\hfill $\blacksquare$}
\newenvironment{proof}{{\bf Proof\\}}{\CQFD}

\def\o{\overline}
\def\a{\alpha}
\def\b{\beta}
\def\Dd{\mathcal D}
\def\Rr{\mathcal R}
\def\Ss{\mathcal S}
\def\Uu{\mathcal U}
\def\Ll{\mathcal L}
\def\eps{\varepsilon}
\def\k{\kappa}
\def\beqn{\begin{eqnarray}}
\def\eeqn{\end{eqnarray}}
\def\ds{\displaystyle}
\def\nn{\nonumber}
\def\grad{{\bf \nabla}}
\def\supp{{\rm supp}}
\def\msc#1{{\bf Mathematical Subject Classification:}~#1}
\def\kwd#1{{\bf Keywords:}~#1}

\begin{document}
\title{On the time continuity of entropy solutions}
\author{Cl\'ement Canc\`es\footnote{UPMC Univ Paris 06, UMR 7598, Laboratoire Jacques-Louis Lions, BC187, 4 place Jussieu F-75005, Paris, France,
 \href{mailto:cances@ann.jussieu.fr}{\tt cances@ann.jussieu.fr}}\ $^,$\footnote{corresponding author. tel: 0033 1 44 27 71 69, fax: 0033 1 44 27 72 00}\ , 
 Thierry Gallou\"et\footnote{LATP, Universit\'e de Provence, 39, rue F. Joliot Curie, 13453 Marseille Cedex 13,  \href{mailto:gallouet@latp.univ-mrs.fr}{\tt gallouet@latp.univ-mrs.fr}}
}

\date{}
\maketitle
\abstract{We show that any entropy solution $u$ of a convection diffusion equation 
$\partial_t u + \div F(u)-\Delta\phi(u) =b$ in $\OT$ belongs to $C([0,T),L^1_{\rm loc}(\o\O))$.
 The proof does not use the uniqueness of the solution.}
 \vskip 10pt
 \msc{35L65, 35B65, 35K65}
 \vskip 10pt 
 \kwd{Entropy solution, time continuity, scalar conservation laws}
\section{The problem, and main result}

Convection diffusion equations appear in a large class of problems, and have been widely studied. 
We consider in the sequel only equations under conservative form:
\be\label{eq_1}
\partial_t u + \div F(u)-\Delta\phi(u) =b,
\ee
so that we can give some sense to \eqref{eq_1} in the distributional sense.
In this paper, we consider entropy solutions of \eqref{eq_1} that do not take into account any 
boundary condition, or condition for $|x|\rightarrow +\infty$. 

The proof does not use a $L^1$-contraction principle (see e.g. Alt \& Luckaus \cite{AL83} or Otto 
\cite{Otto96}), so that it can be applied in case where uniqueness is not insured, like for example 
complex spatial coupling of different conservation laws as in \cite{CGP07}, or for cases where 
uniqueness fails because of boundary conditions or conditions at $|x|=+\infty$, as it will be 
stressed in the sequel. 

Let us now state the required assumptions on the data.
%
%
Let $\O$ be an open subset of $\R^d$ ($d\ge 1$), and let $T$ be a positive real value or $+\infty$.
\be
F  \textrm{ is a continuous function}, \label{F_hyp}\tag{H1}
\ee
\be
\phi  \textrm{ is a nondecreasing Lipschitz function}, \label{phi_hyp}\tag{H2}
\ee
\be
u_0 \in L^1_{\rm loc}(\O). \label{u0_hyp}\tag{H3}
\ee 
One has to make the following assumption on the source term:
\be
b\in L^2_{\rm loc}([0,T);H^{-1}(\O))\cap L^1_{\rm loc}(\O\times[0,T)). \label{b_hyp}\tag{H4}
\ee

In the sequel, $v\top w$ (resp. $v\bot w$) denotes $\max(v,w)$ (resp. $\min(v,w)$), and $\sign$ is the function 
defined by $$\sign(s)=\left\{\begin{array}{lcl}
0 &\textrm{if}&s=0,\\
1 &\textrm{if}&s>0,\\
-1&\textrm{if}&s<0.  
\end{array} \right.$$

We consider entropy weak solutions of \eqref{eq_1}, as
 in the famous work of Kru{\v{z}}kov \cite{K70} for hyperbolic equations.
This notion can be extended to degenerated parabolic equations, as noticed by Carrillo~\cite{Car99}.
This leads to the following definition of entropy weak solution:
\begin{Def}\label{entro_def}
A function $u$ is said to be an entropy weak solution if:
\begin{enumerate}
\item $u\in L^1_{\rm loc}(\O\times[0,T))$,
\item $F(u)\in \left(L^2_{\rm loc}(\O\times[0,T))\right)^d$,
\item $ \phi(u) \in L^2_{\rm loc}([0,T);H^1_{\rm loc}(\O))$,
\item $\forall \psi\in \Dd^+(\O\times[0,T))$, $\forall\k\in\R$, 
\beqn
&\ds \int_0^T \int_\O |u-\k|\partial_t\psi dxdt +\int_\O |u_0-\k|\psi(0)dx &\nn \\
&\ds + \int_0^T \int_\O \left(F(u\top \k)-F(u\bot \k)- \grad |\phi(u)-\phi(\k)|\right)\cdot\grad\psi dxdt&\nn\\
&\ds +\int_0^T\int_\O \textrm{sign}(u-\k) b\psi dxdt\ge 0.&\label{entro_for}
\eeqn
\end{enumerate}
\end{Def}

\begin{prop}\label{entro_weak}
Any entropy weak solution is a weak solution, that is it fulfills the three first points in 
definition~\ref{entro_def}, and: 
$\forall \psi\in \Dd(\O\times[0,T))$,
\beqn
&\ds \int_0^T \int_\O u\partial_t\psi dxdt +\int_\O u_0\psi(0)dx &\nn \\
&\ds + \int_0^T \int_\O \left(F(u)- \grad \phi(u)\right)\cdot\grad\psi dxdt
\ds +\int_0^T\int_\O  b\psi dxdt= 0.&\label{weak_for}
\eeqn
Reciprocally, if $\phi^{-1}$ is a continuous function, the any weak solution is an entropy solution.
\end{prop}
\begin{proof}
Suppose first that $\phi^{-1}$ is a continuous function, then the fact that any weak solution $u$ is 
an entropy weak solution is just based on a convexity inequality, and on the fact that 
$\sign(\phi(a)-\phi(b))=\sign(a-b)$ for all $(a,b)\in \R^2$. More details are available in \cite{Car99}
(see also \cite{GMT94}).\\
The fact that an entropy weak solution $u$ is a weak solution is obvious if $u$ belongs to 
$L^\infty_{\rm loc}( \O\times[0,T) )$ (consider $\k=\pm \| u \|_{L^\infty(\supp(\psi))}$).\\
Suppose now that $u$ only belongs to $L^1_{\rm loc}(\O\times[0,T))$. Let $\k\in\R$, then for all 
$\psi\in \Dd(\O\times[0,T))$, one has
\beqn
&\ds \int_0^T\int_\O \k\partial_t\psi dxdt + \int_\O \k \psi(0)dx=0, \label{k_stat}&
\eeqn
which added to \eqref{entro_for} yields: $\forall \psi\in \Dd^+(\O\times[0,T))$,
\beqn
&\ds \int_0^T \int_\O \left(|u-\k|+\k\right)\partial_t\psi dxdt +\int_\O \left(|u_0-\k|+\k\right)\psi(0)dx &\nn \\
&\ds + \int_0^T \int_\O \sign(u-\k)\left(F(u) -\grad \phi(u)\right)\cdot\grad\psi dxdt&\nn\\
&\ds +\int_0^T\int_\O \textrm{sign}(u-\k) b\psi dxdt\ge 0.&\label{entro_for_k-}
\eeqn
One will now let $\k$ tend to $-\infty$ in \eqref{entro_for_k-}. Suppose that $\k<0$, then
$$
\big||u-\k|+\k \big| \le |u| \textrm{ and } \big||u-\k|+\k \big|\rightarrow u \textrm{ a.e. in } \supp(\psi),
$$
and the dominated convergence theorem gives: $\forall \psi\in \Dd^+(\O\times[0,T))$,
\beqn
&\ds \int_0^T \int_\O u\partial_t\psi dxdt +\int_\O u_0\psi(0)dx &\nn \\
&\ds + \int_0^T \int_\O \left(F(u)- \grad \phi(u)\right)\cdot\grad\psi dxdt
\ds +\int_0^T\int_\O  b\psi dxdt\ge  0.&\nn\label{weak_for_+}
\eeqn
The same way, one has: $\forall \psi\in \Dd^+(\O\times[0,T))$,
\beqn
&\ds \int_0^T \int_\O \left(|u-\k|-\k\right)\partial_t\psi dxdt +\int_\O \left(|u_0-\k|-\k\right)\psi(0)dx &\nn \\
&\ds + \int_0^T \int_\O \sign(u-\k)\left(F(u) -\grad \phi(u)\right)\cdot\grad\psi dxdt&\nn\\
&\ds +\int_0^T\int_\O \textrm{sign}(u-\k) b\psi dxdt\ge 0.&\nn\label{entro_for_k+}
\eeqn
Letting $\k$ tend to $+\infty$, one gets: $\forall \psi\in \Dd^+(\O\times[0,T))$,
\beqn
&\ds \int_0^T \int_\O u\partial_t\psi dxdt +\int_\O u_0\psi(0)dx &\nn \\
&\ds + \int_0^T \int_\O \left(F(u)- \grad \phi(u)\right)\cdot\grad\psi dxdt
\ds +\int_0^T\int_\O  b\psi dxdt\le  0.&\nn
\eeqn
This insures that: $\forall \psi\in \Dd^+(\O\times[0,T))$,
\beqn
&\ds \int_0^T \int_\O u\partial_t\psi dxdt +\int_\O u_0\psi(0)dx &\nn \\
&\ds + \int_0^T \int_\O \left(F(u)- \grad \phi(u)\right)\cdot\grad\psi dxdt
\ds +\int_0^T\int_\O  b\psi dxdt =  0.&\label{weak_for_1}
\eeqn
It is now easy to check that \eqref{weak_for_1} still holds for $\psi\in  \Dd(\O\times[0,T))$, and so this 
achieves the proof of propostion~\ref{entro_weak}
\end{proof}
\begin{rmk}\label{F_L1}
In the case where $\phi\equiv 0$, the point 2 of definition \ref{entro_def} can be replaced 
by 
$$
F(u)\in \left(L^1_{\rm loc}(\O\times[0,T))\right)^d,
$$
and one can remove the assumption $b\in L^2_{\rm loc}([0,T);H^{-1}(\O))$ in \eqref{b_hyp}.
Actually, in such a case, Kru\v{z}kov entropies $|\cdot - \k|$ are sufficient to obtain the time continuity. 
The assumptions  $F(u)\in \left(L^2_{\rm loc}(\O\times[0,T))\right)^d$ and
 $b\in L^2_{\rm loc}([0,T);H^{-1}(\O))$
will only be useful to insure $\partial_t u$ belongs to $L^2_{\rm loc}([0,T);H^{-1}(\O))$ in order to
recover the regular convex entropies, which are necessary to treat the parabolic case, as it 
was shown in the work of Carrillo~\cite{Car99}. 
\end{rmk}

The definition~\ref{entro_def} does not take into account any boundary condition, or condition at 
$|x| \rightarrow +\infty$. This lack of regularity can lead to non-uniqueness cases, as the one shown in the 
book of Friedman~\cite{F64} (also available in the one of Smoller~\cite{Smo94}): the very simple problem 
\be\label{SMO}
\left\{\begin{array}{ll}
\partial_t u - \partial^2_{xx} u =0 &\textrm{ in } \R\times \R_+,\\
u(\cdot, 0)=0 &\textrm{ in } \R
\end{array}\right.
\ee
admits multiple classical solutions if one does not ask some condition for large $x$ like e.g. 
$u\in \Ss'(\R\times\R_+)$. Indeed, it is easy to check that
$$
u(x,t)=\sum_{k=0}^{\infty} \frac{1}{2k !}x^{2k}\frac{d^k}{dt^k}e^{-1/t^2}
$$
is a classical solution of \eqref{SMO}. So $u$ is a weak solution of \eqref{SMO}, 
and thus an entropy weak solution thanks to proposition~\ref{entro_weak}.  
It also belongs to $C([0,T],L^1_{\rm loc}(\R))$, thanks to its regularity. 

Let us give another example, proposed by Michel Pierre~\cite{Exemple_Michel}. We now consider the problem
\be\label{MP}
\left\{\begin{array}{ll}
\partial_t u - \partial^2_{xx} u =0 &\textrm{ in } [0,1]\times \R_+,\\
u(\cdot, 0)=0 &\textrm{ in } [0,1],\\
u(0,\cdot) = u(1,\cdot) = 0 & \textrm{ in } \R_+
\end{array}\right.
\ee
which admits the constant function equal to $0$ as unique smooth solution. A non-smooth solution to the problem~\eqref{MP} can be built as follows.
Denote by $u_f$ the fundamental solution of the heat equation in the one-dimensional case:
$$
u_f(x,t) = \frac{1}{\sqrt{4\pi t}} \exp\left( -\frac{x^2}{4t} \right), 
$$
then $v:=\partial_x u_f$ also satisfies the heat equation in the distributional sense. The function $v$, given by 
$$
v(x,t) = -\frac{2x}{t \sqrt{4\pi t}} \exp\left( -\frac{x^2}{4t} \right),
$$
satisfies $v(0,t) = 0$ for all $t > 0$, 
belongs to ${C}^\infty\left( [0,1]\times[0,T]\setminus\{(0,0)\}  \right)$ but is not continuous in $(x,t) = (0,0)$. Indeed, one has
$$
\lim_{s\to 0^+} v(\sqrt s, s) = -\infty.
$$
The function $t\mapsto  v(1,t)$ belongs to $C^\infty(\R_+)$, then there exists a unique $w\in {C}^\infty([0,1]\times\R_+)$ solution to the problem
$$
\left\{\begin{array}{ll}
\partial_t w - \partial^2_{xx} w =0 &\textrm{ in } [0,1]\times \R_+,\\
w(\cdot, 0)=0 &\textrm{ in } [0,1],\\
w(0,\cdot)  = 0 & \textrm{ in } \R_+,\\
w(1,t) = v(1,t) & \textrm{ in } \R_+.\\
\end{array}\right.
$$
Defining $u:=v-w$, then $u$ is a solution to the problem~\eqref{MP} which is not the trivial solution since it is not regular. Nevertheless, $u$ is a weak solution to the problem and thus a entropy weak solution thanks to proposition~\ref{entro_weak}. Thanks to its regularity, it clearly appears that $u$ belongs 
to $C(\R_+;L^1_{\rm loc}((0,1))$.

In the following theorem, we claim that any entropy solution is time continuous
with respect with the time variable, 
at least locally with respect to the space variable.
\begin{thm}\label{main}
Let $u$ be a entropy solution in the sense of definition~\ref{entro_def}, then there exists 
$\o{u}$ such that $u=\o{u}$ a.e. on $\O\times[0,T)$ and fulfilling
$$\o{u}\in C([0,T);L^1_{\rm loc}(\O)).$$
Furthermore, if there exists $p>1$ and a neighborhood $\Uu$ of $\partial\O$ in $\O$ such that
$$u_0\in L^p_{\rm loc}(\Uu), \qquad u\in L^\infty_{\rm loc}([0,T);L^p_{\rm loc}(\Uu)),$$
then we have:
$$\o{u}\in C([0,T);L^1_{\rm loc}(\o\O)).$$
\end{thm}

\section{Essential continuity for $t=0$}
In this section, we give a simple way to prove the classical result stated in proposition~\ref{t0_prop}. 
\begin{Def}\label{Lebesgue_def}
One says that $t\in [0,T)$ is a right-Lebesgue point if there exists $\o u(t)$ in $L^1_{\rm loc}(\O)$ such that 
for all compact subset $K$ of $\O$, 
$$
 \lim_{\eps\rightarrow 0}\frac{1}{\eps}\int_t^{t+\eps} \|u(s)-\o{u}(t)\|_{L^1(K)}ds = 0.
$$
We denote by $\Ll$ the set of right-Lebesgue points. 
\end{Def}
It is well known that $meas\left((0,T)\setminus \Ll\right)=0$ and that $u=\o u$ (in the 
$L^1_{\rm loc}(\O)$-sense) a.e. in $(0,T)$.
In the sequel, we will prove that $\Ll=[0,T)$, and that $\o{u}$ belongs to $C([0,T);L^1_{\rm loc}(\O))$. 
We begin by considering the essential continuity for the initial time $t=0$.
\begin{prop}\label{t0_prop}
For all $\zeta\in \Dd^+(\O)$, one has:
$$
\lim_{\begin{array}{c} t\rightarrow 0\\ t\in \Ll\end{array}}
\int_\O |\o u(x,t)-u_0(x)|\zeta(x)dx = 0.
$$
Particularly, this ensures that $0\in \Ll$.
\end{prop}
The limit as $t$ tends to $0$, $t\in \Ll$ can 
be seen as an essential limit, as it is done in lemma~7.41 in the book of M\`alek et al. \cite{MNRR96} in the case 
of a purely hyperbolic problem, or by Otto \cite{Otto96} in the case of a non strongly degenerated parabolic equation. 
See also the paper of Blanchard and Porretta \cite{BP05} for the case of renormalized solutions 
for degenerate parabolic equations.\\
\begin{proof}
First, notice that for all $t\in \Ll$, and for all $\k\in \R$, $t$ is also a right-hand side Lebesgue point of $|u-\k|$.
Indeed, if $K$ denotes a compact subset of $\o\O$, one has for a.e $(x,s)\in \O\cap K\times (0,T)$
$$
\big| |u(x,s)-\k|- |u(x,t)-\k| \big| \le |u(x,s)-u(x,t)|,
$$
and so, for all $t\in \Ll$, 
\be\label{Leb2}
\lim_{\a\rightarrow 0}\frac{1}{\a}\int_t^{t+\a} \int_{\O\cap K} \big| |u(x,s)-\k|- |u(x,t)-\k| \big| dx ds =0.
\ee
Let $\a>0$, and $t^\star\in \Ll$, one denotes 
$$\chi^\a_{[0,t^\star[}(t)=\left\{ \begin{array}{ll} 
1 &\textrm{ if }t\le t^\star\\ 
0 &\textrm{ if }t\ge t^\star+\a \\
\frac{t^\star+\a-t}{\a}&\textrm{ if }t^\star< t <t^\star+\a.
\end{array}\right.$$

Let $\zeta\in \Dd(\O)$, and let $\eps>0$ be such that $d(supp(\zeta),\partial\O)>\eps$. Let $\rho\in \Dd^+(\R^d)$, with 
$supp(\rho)\subset B(0,1)$ and $\int_{\R^d} \rho(z)dz =1$. One denotes 
$\rho_\eps(z)= \frac{1}{\eps^d} \rho(\frac{z}{\eps})$. The function $y\mapsto \zeta(x)\rho_\eps(x-y)$ belongs to $\Dd^+(\O)$.\\
Taking $\k= u_0(y)$ and $\psi(x,y,t)=\zeta(x)\rho_\eps(x-y) \chi^\a_{[0,t^\star[}(t)$ in (\ref{entro_for}), an integrating with respect to $y\in \O$ 
yields:
\beqn
&\ds \int_0^T \int_\O\int_\O |u(x,t)-u_0(y)|\zeta(x)\rho_\eps(x-y)\partial_t  \chi^\a_{[0,t^\star[}(t)dxdydt &\nn\\
&\ds +\int_\O\int_\O  |u_0(x)-u_0(y)|\zeta(x)\rho_\eps(x-y)dxdy &\nn \\
&\ds + \int_0^T  \chi^\a_{[0,t^\star[}(t) \int_\O \int_\O\left[\begin{array}{c}(F(u(x,t)\top u_0(y))-F(u(x,t)\bot u_0(y))) \\
\cdot
\grad \left(\zeta(x)\rho_\eps(x-y)\right)\end{array}\right]dxdydt& \nn\\
&\ds -\int_0^T  \chi^\a_{[0,t^\star[}(t) \int_\O \int_\O \grad |\phi(u(x,t))-\phi(u_0(y))|\cdot\grad \left(\zeta(x)\rho_\eps(x-y)\right) dxdydt&\nn\\
&\ds +\int_0^T \chi^\a_{[0,t^\star[}(t) \int_\O\int_\O  \left[\begin{array}{c}\textrm{sign}(u(x,t)-u_0(y)) b(x,t)\\
\zeta(x)\rho_\eps(x-y) \end{array}\right]dxdydt\ge 0,&\label{t0_1}
\eeqn
where all the gradient are considered with respect to $x$, and not $y$.

One has $$|u(x,t)-u_0(y)| = |u(x,t)-u_0(x)| + |u(x,t)-u_0(y)|-|u(x,t)-u_0(x)|,$$ then, since 
$\int_{\R^d}\rho_\eps(x-y)dy=1$ for all $x$ in $\supp (\zeta)$,  
using  $$|u_0(x)-u_0(y)|\ge \big| |u(x,t)-u_0(y)|-|u(x,t)-u_0(x)|\big|,$$ we obtain
\beqn
&\ds  \int_0^T\partial_t  \chi^\a_{[0,t^\star[}(t) \int_\O\int_\O |u(x,t)-u_0(y)|\zeta(x)\rho_\eps(x-y)dxdydt &\nn\\
&\ds \le  \int_0^T\partial_t  \chi^\a_{[0,t^\star[}(t) \int_\O|u(x,t)-u_0(x)|\zeta(x) dxdt &\nn\\
& \ds + \|\partial_t  \chi^\a_{[0,t^\star[}\|_{L^1(0,T)} \int_\O\int_\O  |u_0(x)-u_0(y)|\zeta(x)\rho_\eps(x-y)dxdy. &\label{t0_2}
\eeqn
For all $\a\in ]0,T-t^\star]$, 
$$
 \|\partial_t  \chi^\a_{[0,t^\star[}\|_{L^1(0,T)}=1,
$$
and then, one can let $\a$ tend to $0$ in (\ref{t0_2}), so that (\ref{t0_1}) implies:
\beqn
&\ds -  \int_\O\int_\O |\o u(x,t^\star)-u_0(x)|\zeta(x)dxdy   &\nn\\
&\ds  + 2   \int_\O\int_\O |u_0(x)-u_0(y)|\zeta(x)\rho_\eps(x-y)dxdy   
\ds + \int_0^{t^\star} \Rr_\eps(t) dt \ge 0,&\label{t0_3}
\eeqn
where $\Rr_\eps$ belongs to $L^1(0,T)$ for all $\eps>0$.  Since $\Ll$ is dense in 
$[0,T]$, one can let in a first step $t^\star$ tend to $0$, so that $\int_0^{t^\star} \Rr_\eps(t)dt$ vanishes:
\beqn
&\ds \limsup_{\begin{array}{c} t^\star \rightarrow 0 \\ t^\star\in \Ll  \end{array}} 
\int_\O\int_\O |\o u(x,t^\star)-u_0(x)|\zeta(x)dxdy  &\nn\\
&\ds \le 2   \int_\O\int_\O |u_0(x)-u_0(y)|\zeta(x)\rho_\eps(x-y)dxdy .&\label{t0_4}
\eeqn
One can now let $\eps$ tend to $0$, and using the fact that $u_0$ belongs to $L^1_{\rm loc}(\O)$, and that $\zeta$ is compactly supported in $\O$, one gets:
$$
 \lim_{\begin{array}{c} t^\star \rightarrow 0 \\ t^\star\in \Ll  \end{array}} 
 \int_\O\int_\O |\o u(x,t^\star)-u_0(x)|\zeta(x)dxdy=0.
$$
 This achieves the proof of proposition~\ref{t0_prop}.
\end{proof}
\section{Time continuity for any $t\ge 0$}
In this section, we want to prove the following proposition:
\begin{prop}\label{tge0_prop}
Let $u$ be a entropy solution in the sense of definition~\ref{entro_def}, then there exists 
$\o{u}$ such that $u=\o{u}$ a.e. on $\O\times(0,T)$ and fulfilling
$$\o{u}\in C([0,T);L^1_{\rm loc}(\O)).$$
\end{prop}
In the sequel, we still denote by $\o u$ the representative defined using the right Lebesgue points 
introduced in definition~\ref{Lebesgue_def}.
Proving the essential continuity for every $t^\star\in \Ll$ is easy.
Indeed, if one replaces $\psi(x,t)$ by $(1-\chi^\a_{[0,t^\star[})(t)\psi(x,t)$ in (\ref{entro_for}), and then if one lets $\a$
tend to $0$, one gets:
\beqn
&\ds \int_{t^\star}^T \int_\O |u-\k|\partial_t\psi dxdt +\int_\O |\o u({t^\star})-\k|\psi({t^\star})dx &\nn \\
&\ds + \int_{t^\star}^T \int_\O \left(F(u\top \k)-F(u\bot \k)- \grad |\phi(u)-\phi(\k)|\right)\cdot\grad\psi dxdt&\nn\\
&\ds +\int_{t^\star}^T\int_\O \textrm{sign}(u-\k) b\psi dxdt\ge 0.&\label{entro_for_tstar}
\eeqn
One can thus apply the proposition~\ref{t0_prop} with $t^\star$ instead of $0$, and $\o u(t^\star)$ instead of $u_0$: $\forall \zeta\in \Dd^+(\O)$,
$$
 \lim_{\begin{array}{c} s^\star \rightarrow t^\star \\ s^\star\in \Ll  \end{array}} \int_\O\int_\O |\o u(x,s^\star)-
 \o u(x,t^\star)|\zeta(x)dxdy=0.
$$
We will prove the uniform continuity of $t\mapsto \o u(t)$ from $\Ll\cap[0,T-\gamma]$ to $L^1_{\rm loc}(\O)$ 
for all $\gamma\in (0,T)$. This will give as a direct consequence that $\Ll=[0,T)$ and 
$\o u\in C([0,T);L^1_{\rm loc}(\O))$.  This uniform continuity will
 come from theorem 13 in the paper of Carrillo \cite{Car99}, which, adapted to our case, can be stated as follow:
\begin{thm}\label{Car_thm}
Suppose that \eqref{F_hyp}, \eqref{phi_hyp} hold.  Let $u_0, v_0$ belong to $L^1_{\rm loc}(\O)$, let $b_u, b_v$ belong to 
$ L^2((0,T);H^{-1}(\O))\cap L^1((0,T);L^1_{\rm loc}(\O))$, 
 and let $u,v$ be two entropy solutions associated to the choice of $b=b_u$ and initial data $u_0$ for 
 $u$ and $b=b_v$ and initial data $v_0$ for $v$ in definition~\ref{entro_def}. 
 Then $\forall\psi\in \Dd^+(\O\times[0,T[)$, 
\beqn
&\ds \int_0^T\int_\O | u - v |\partial_t \psi dxdt + \int_\O |u_0 -v_0| \psi(0) dx&\nn \\
&\ds \int_0^T \int_\O \left(F(u\top v)-F(u\bot v)- \grad |\phi(u)-\phi(v)|\right)\cdot\grad\psi dxdt&\nn\\
&\ds +\int_0^T\int_\O \textrm{sign}(u-v) (b_u-b_v)\psi dxdt\ge 0.&
\label{Car_for}\eeqn
\end{thm}
We now have all the tools for the proof of proposition~\ref{tge0_prop}.\\
{\bf Proof of proposition~\ref{tge0_prop}}\\
Let $\gamma>0$, let $t^\star\in \Ll_\gamma=\Ll\cap[0,T-\gamma]$, and $h\in\Ll_\gamma$ such that $t^\star+ h\in \Ll_\gamma$ (this is the case of almost every $h\in (0,T-t^\star-\gamma)$).
Let $\zeta\in \Dd^+(\O)$, let $\a\in]0,T-t^\star-\gamma-h[$.
   
Taking $\psi(x,t)=\zeta(x)\chi_{[0,t^\star[}^\a(t)$, $v_0(x)=u(x,h)$, $v(x,t)=v(x,t+h)$ in \eqref{Car_for}, and letting $\a$ tend to $0$ yields:
\beqn
&\ds -\int_\O |\o u(x,t^\star) - \o u(x,t^\star+h) |\zeta(x) dx + \int_\O |u_0(x) -\o u(x,h)| \zeta(x) dx &\nn \\
&\ds \int_0^{t^\star}\!\! \int_\O \left[\begin{array}{c}\!\!
F(u(x,t)\top u(x,t+h))-F(u(x,t)\bot u(x,t+h))\!\!  \\
- \grad |\phi(u(x,t))-\phi(u(x,t+h))|\end{array}\right]\cdot\grad\zeta(x) dxdt\!\!\!\!&\nn\\
&\ds +\int_0^{t^\star}\int_\O 
\left[
\begin{array}{c}\textrm{sign}(u(x,t)-u(x,t+h)) \\(b(x,t)-b(x,t+h))
\end{array}\right]
\zeta(x) dxdt\ge 0.&
\label{Car_for2}\eeqn
We deduce from \eqref{Car_for2} that 
\beqn
&\ds \int_\O |\o u(x,t^\star) -\o u(x,t^\star+h) |\zeta(x) dx \le \int_\O |u_0(x) - \o u(x,h)| \zeta(x) dx&\nn\\
&\ds\!\!\!\!\!\!\!+ \! \int_0^{T-\gamma-h}\!\!\! \int_\O \left| F(u(x,t)\top u(x,t+h))-
F(u(x,t)\bot u(x,t+h)) \right| |\grad \zeta(x)| dxdt \nn \!\!\!\!\!\!\!\!& \\
& \ds+  \int_0^{T-\gamma-h} \int_\O \left|\grad\phi(u)(x,t+h)-\grad\phi(u)(x,t)\right|
 |\grad \zeta(x)| dxdt \nn & \\
&\ds+  \int_0^{T-\gamma-h} \int_\O \left|b(x,t+h)-b(x,t)\right|  \zeta(x) dxdt,&\nn 
\eeqn
and since $F(u), \grad\phi(u)$ and $b$ belong to $L^1_{\rm loc}(\O\times(0,T))$, one can claim that:
\beqn
&\forall \eps >0, \forall t^\star\in \Ll_\gamma, \exists \eta >0\ s.t.\ \forall h\in \Ll\cap[0,T-\gamma-t^\star], 
h\le \eta \Rightarrow \nn&\\
&\ds \!\!\!\!\!\!\! \int_\O |\o u(x,t^\star) -\o u(x,t^\star+h) |\zeta(x) dx \le \int_\O |u_0(x) -\o u(x,h)|
 \zeta(x) dx + \eps.&
\label{unif_cont1}
\eeqn
One can now use proposition~\ref{t0_prop} in \eqref{unif_cont1}, so that we get that 
$$
t\mapsto \o u(x,t) \textrm{ is uniformly continuous from }\Ll \textrm{ to } L^1(\O,\zeta),
$$
which is the $L^1$-space for measure of density $\zeta$ w.r.t. Lebesgue measure.
We deduce that, for all $\gamma\in (0,T)$, $t\mapsto \o u$ is uniformly continuous from 
$\Ll_\gamma$ to $L^1_{\rm loc}(\O)$, and this insures that $\Ll_\gamma=[0,T-\gamma]$.
This holds for any $\gamma\in (0,T)$, and so we can claim that 
$\o u \in C([0,T);L^1_{\rm loc}(\O))$.
\hfill$\blacksquare$

It remains to prove the last part of theorem~\ref{main} by considering some test functions 
$\zeta\in \Dd^+(\o\O)$ instead of $\zeta\in \Dd^+(\O)$. We will need some additional regularity on the solution:
\be\left\{\begin{array}{c}
\textrm{There exists an open neighborhood }\Uu \textrm{ of }\partial\O \textrm{ in }\o\O \textrm{ s.t. }\\
u_0\in L^p_{\rm loc}(\Uu), \qquad u\in L^\infty_{\rm loc}([0,T);L^p_{\rm loc}(\Uu)).
\end{array}\right\}\label{H_bord}\tag{H5}\ee
\eqref{H_bord} gives the uniform (w.r.t. $t$) local equiintegrability of $u$ (and so of $\o u$) on a 
neighborhood of $\Uu$. We deduce, using $\o u\in C([0,T);L^1_{\rm loc}(\O))$ that 
$\o u\in C([0,T);L^1_{\rm loc}(\o \O))$.\\
{\bf End of the proof of theorem~\ref{main}}\\
Suppose that \eqref{F_hyp},\eqref{phi_hyp},\eqref{u0_hyp},\eqref{b_hyp} hold, then thanks to 
proposition~\ref{tge0_prop}, there exists a weak solution $\o u\in C([0,T),L^1_{\rm loc}(\O))$. \\
For $\eps>0$, $\gamma\in (0,T)$, $\zeta\in \Dd^+(\O)$, there exists $\eta>0$ such that: 
$\forall t\in [0,T-\gamma]$, $\forall h\in [0, \min(\eta, T-t-\gamma)]$, 
$$
\int_\O |\o u(x,t+h)-\o u(x,t)| \zeta(x)dx \le \eps.
$$
Let $K$ be a compact subset of $\o\O$. Then there exists $\zeta\in \Dd^+(\o\O)$ such that 
$0\le \zeta(x) \le 1$ for all $x\in \R^d$, and
$\zeta(x)=1$ if $x\in \o\O$. Let $\a>0$ and let $\beta_\a \in C^\infty(\R^d;\R)$ such that:
\begin{eqnarray*}
&&0\le \beta_\a(x) \le 1\qquad \textrm{ for all } x\in \R^d,\\
&&\beta_\a(x) = 1 \qquad\qquad \textrm{ if } d(x, \partial\O)\le \a/2,\\
&& \beta_\a(x) = 0 \qquad \qquad\textrm{ if } d(x,\partial\O)\ge \a.  
\end{eqnarray*}
Suppose that \eqref{H_bord} holds. 
For $\a$ small enough, one has $supp(\zeta\b_\a) \subset \Uu$ and then, for all $t\in [0,T-\gamma]$, 
for all $h\in [0,T-t-\gamma]$, 
$$
\int_\O |\o u(x,t+h)-\o u(x,t)| \zeta(x)\beta_\a dx\le 2 \|u\|_{L^\infty((0,T-\gamma);L^p(\Uu_\zeta))}
\|\b_\a\|_{L^{p'}(\Uu_\zeta)},
$$
where $\Uu_\zeta$ denotes $\Uu\cap supp(\zeta)$, and $p'=\frac{p}{p-1}<+\infty$.
Since $\|\b_\a\|_{L^{p'}(\Uu_\zeta)}$ tends to $0$ as $\a$ tends to $0$, there exists $\delta>0$ such 
that: 
\be\label{conclu1}
\a\le \delta \Rightarrow \int_\O |\o u(x,t+h)-\o u(x,t)| \zeta(x)\beta_\a dx\le \eps.
\ee
Suppose now that $\a$ has been chosen such that \eqref{conclu1} holds. The function $\zeta(1-\b_\a)$ 
belongs to $\Dd^+(\O)$, and then there exists $\eta$ such that :$\forall t\in [0,T-\gamma]$, 
$\forall h\in [0, \min(\eta, T-\gamma-t)]$, 
\be\label{conclu2}
\int_\O |\o u(x,t+h)-\o u(x,t)| \zeta(x)(1-\beta_\a(x))dx \le \eps.
\ee
Adding \eqref{conclu1} and \eqref{conclu2} shows that for all $t$ in $[0,T-\gamma-\eta]$, 
for all $h\in [0,\eta]$,
\be\label{conclu3}
\int_K |\o u(x,t+h)-\o u(x,t)| dx \le 2\eps.
\ee
So $\o u$ is uniformly continuous from $[0,T-\gamma]$ to $L^1(K)$, and then
$$
\o u \in C([0,T);L^1_{\rm loc}(\o\O)).
$$
\hfill$\blacksquare$

To conclude this paper, let us give a counter-example to the time continuity in the case where the 
entropy criterion is not fulfilled for t=0. Consider the Burgers equation, in the one dimensional case, 
leading to the following initial value problem.
\be\label{burgers}
\left\{\begin{array}{l}
\partial_t u - \partial_x \left(  u^2 \right) =0, \quad (x,t)\in (\R\times \R_+),\\
u(\cdot,0)= u _0 = 0.
\end{array}\right.
\ee
Problem \eqref{burgers} admits $u=0$ as unique entropy solution in the sense of 
definition~\ref{entro_def}.

We define 
$$\tilde{u}(x,t)= \left\{\begin{array}{rcl}
0 & \textrm{if} & t=0,\\
0& \textrm{if} & |x|>\sqrt{t},\\
\ds \frac{x}{2t} & \textrm{if} &|x|<\sqrt{t}.
\end{array}
\right.$$
Then it is easy to check that:
\begin{itemize}
\item $\tilde{u}\in L^1_{\rm loc}(\R\times \R_+)$,
\item $\tilde{u}^2\in L^1_{\rm loc}(\R\times \R_+)$,
\item $\forall \psi\in \Dd(\O\times \R_+)$, 
\beqn
&\ds \int_0^{+\infty}\int_\R \tilde{u}(x,t)\partial_t \psi (x,t)dxdt &\nn\\
&\ds+ \int_0^{+\infty}\int_\R \tilde{u}^2(x,t) \partial_x \psi(x,t)dxdt =0,&\label{burgers_weak}
\eeqn
\item $\forall \psi\in \Dd^+(\O\times \R_+^\star)$, $\forall \k\in \R$, 
\beqn
&\ds \int_0^{+\infty}\int_\R |\tilde{u}-\k|(x,t)\partial_t \psi (x,t)dxdt&\nn\\
&\ds + \int_0^{+\infty}\int_\R \sign(\tilde{u}-\k)\tilde{u}^2(x,t) \partial_x \psi(x,t)dxdt =0.
&\label{burgers_entro}
\eeqn
\end{itemize}
Thanks to \eqref{burgers_weak}, $\tilde{u}$ is a weak solution of \eqref{burgers}, and an 
entropy criterion \eqref{burgers_entro} is fulfilled only for $t>0$. The fact that the entropy 
criterion fails for $t>0$, and that the flux $\tilde{u}^2$ is not bounded (see \cite{CR00}) 
allows the function $\tilde{u}$ to be discontinuous  at $t=0$. Indeed, for all $t>0$, 
$$
\| \tilde u(\cdot,t) \|_{L^1(\R)} = \frac{1}{2} \neq \|u_0 \|_{L^1(\R)}=0.
$$

\bibliographystyle{plain}

\begin{thebibliography}{10}

\bibitem{AL83}
H.~W. Alt and S.~Luckhaus.
\newblock Quasilinear elliptic-parabolic differential equations.
\newblock {\em Math. Z.}, 183(3):311--341, 1983.

\bibitem{BP05}
D.~Blanchard and A.~Porretta.
\newblock Stefan problems with nonlinear diffusion and convection.
\newblock {\em J. Differential Equations}, 210(2):383--428, 2005.

\bibitem{CGP07}
C.~Canc\`es, T.~Gallou\"et, and A.~Porretta.
\newblock Two-phase flows involving capillary barriers in heterogeneous porous
  media.
\newblock to appear in Interfaces Free Bound., 2009.

\bibitem{Car99}
J.~Carrillo.
\newblock Entropy solutions for nonlinear degenerate problems.
\newblock {\em Arch. Ration. Mech. Anal.}, 147(4):269--361, 1999.

\bibitem{CR00}
G.~Chen and M.~Rascle.
\newblock Initial layers and uniqueness of weak entropy solutions to hyperbolic
  conservation laws.
\newblock {\em Arch. Ration. Mech. Anal.}, 153(3):205--220, 2000.

\bibitem{F64}
A.~Friedman.
\newblock {\em Partial differential equations of parabolic type}.
\newblock Prentice-Hall Inc., Englewood Cliffs, N.J., 1964.

\bibitem{GMT94}
G.~Gagneux and M.~Madaune-Tort.
\newblock Unicit\'e des solutions faibles d'\'equations de
  diffusion-convection.
\newblock {\em C. R. Acad. Sci. Paris S\'er. I Math.}, 318(10):919--924, 1994.

\bibitem{K70}
S.~N. Kru{\v{z}}kov.
\newblock First order quasilinear equations with several independent variables.
\newblock {\em Mat. Sb. (N.S.)}, 81 (123):228--255, 1970.

\bibitem{MNRR96}
J.~M{\'a}lek, J.~Ne{\v{c}}as, M.~Rokyta, and M.~R{\ocirc{u}}{\v{z}}i{\v{c}}ka.
\newblock {\em Weak and measure-valued solutions to evolutionary {PDE}s},
  volume~13 of {\em Applied Mathematics and Mathematical Computation}.
\newblock Chapman \& Hall, London, 1996.

\bibitem{Otto96}
F.~Otto.
\newblock ${L}^1$-contraction and uniqueness for quasilinear elliptic-parabolic
  equations.
\newblock {\em J. Differential Equations}, 131:20--38, 1996.

\bibitem{Exemple_Michel}
M.~Pierre.
\newblock Personal discussion.

\bibitem{Smo94}
J.~Smoller.
\newblock {\em Shock waves and reaction-diffusion equations}, volume 258 of
  {\em Fundamental Principles of Mathematical Sciences}.
\newblock Springer-Verlag, New York, second edition, 1994.

\end{thebibliography}
\def\ocirc#1{\ifmmode\setbox0=\hbox{$#1$}\dimen0=\ht0 \advance\dimen0
  by1pt\rlap{\hbox to\wd0{\hss\raise\dimen0
  \hbox{\hskip.2em$\scriptscriptstyle\circ$}\hss}}#1\else {\accent"17 #1}\fi}

\end{document}